\documentclass{article}
\usepackage{amsmath}
\usepackage{amsfonts}
\usepackage{amsthm}
\usepackage{amscd}
\usepackage{amssymb}
\usepackage{color}
\usepackage{enumerate}

\newtheorem{theorem}{Theorem}[section]
\newtheorem{lemma}[theorem]{Lemma}

\newtheorem{note}[theorem]{Note}
\newtheorem{proposition}[theorem]{Proposition}

\newtheorem{observation}[theorem]{Observation}
\newtheorem{problem}[theorem]{Problem}

\newtheorem{corollary}[theorem]{Corollary}

\newtheorem{question}[theorem]{Question}

\theoremstyle{definition}
\newtheorem{definition}[theorem]{Definition}
\newtheorem{example}[theorem]{Example}
\theoremstyle{remark}
\newtheorem{remark}[theorem]{Remark}

\newcommand{\1}{\mathbf{1}}

\def\1{\mathbf{1}}

\def\mv{\mathcal{V}}
\def\mju{\mathcal{U}}

\def\f2{\mathbb{F}_2}

\def\lip{\hskip0.02cm{\rm Lip}\hskip0.01cm}

\newcommand{\ep}{\varepsilon}

\newcommand{\bbN}{\mathbb{N}}
\newcommand{\al}{\alpha}
\newcommand{\be}{\beta}

\newcommand{\e}{\varepsilon}

\newcommand{\la}{\lambda}

\newcommand{\s}{\sigma}
\newcommand{\vf}{\varphi}

\newcommand{\om}{\omega}

\newcommand{\disp}{\displaystyle}
\newcommand{\lb}{\label}

\newcommand{\wtw}{if and only if}

\newcommand{\DEF}{\buildrel {\mbox{\tiny def}}\over =}

\newcommand\remove[1]{}

\numberwithin{equation}{section}

\begin{document}

\title{On $L_1$-embeddability of unions of $L_1$-embeddable metric spaces and of twisted unions of hypercubes}

\author{Mikhail I.~Ostrovskii and Beata Randrianantoanina}

\date{\today}
\maketitle

\begin{large}

\begin{abstract}
We study properties of twisted unions of metric spaces
introduced by Johnson, Lindenstrauss, and Schechtman, and by Naor
and Rabani. In particular, we prove that under certain natural mild assumptions
 twisted unions of $L_1$-embeddable metric spaces also embed in $L_1$ with distortions
 bounded above by constants that do not depend on the metric spaces themselves, or on their size, but only on certain general parameters. This answers a question stated  by Naor
and  by Naor
and Rabani.

In the second part of the paper we give new
simple examples of metric spaces such their every embedding into  $L_p$, $1\le
p<\infty$, has distortion at least $3$, but which are a union of two subsets, each isometrically embeddable in $L_p$. This extends an analogous result of K.~Makarychev and  Y.~Makarychev from Hilbert spaces to $L_p$-spaces, $1\le
p<\infty$.
\end{abstract}

\smallskip



\section{Introduction}\label{S:Intro}

One of 
natural general questions about 
metric
spaces is the following:

\begin{question}\label{Q:Gen} Let a metric space $(X,d)$ be a union of its metric subspaces $A$ and $B$. Assume that $A$ and $B$ have
a certain metric property $\mathcal{P}$. Does this imply that  $X$ also has
property $\mathcal{P}$, possibly in some weakened form?
\end{question}

This question can be viewed as a part of a general theme of ``local-global'' properties, when one wants to analyze whether spaces (or other mathematical objects) that have certain properties ``locally'', i.e. on certain subspaces/subsets, also have related properties ``globally'', i.e. on the whole space. The study of the ``local-global'' theme is prevalent in many (if not all) areas of mathematics, including functional analysis, and of theoretical computer science. Questions in the ``local-global'' theme usually assume that {\it all} subspaces/subsets of a specified size satisfy the investigated property.  Question~\ref{Q:Gen}  is different since it assumes that property
$\mathcal{P}$ is satisfied by only one pair of complementing subsets, at least one of which has to be at least half the size of $X$.

We are particularly interested in the embeddability properties of
metric spaces. We are aware of three embeddability properties for
which the answers to  Question~\ref{Q:Gen}   are positive,
interesting, and useful. We  state  them below after recalling the
necessary definitions.


\begin{definition}\label{D:Dist} Let $(X,d_X)$ and $(Y,d_Y)$ be metric spaces. An injective map  $F: X\to Y$ is called a {\it bilipschitz embedding} if there exist constants $C_1,C_2>0$ so that for all $u,v\in X$
\[C_1d_X(u,v)\le d_Y(F(u),F(v))\le C_2d_X(u,v).\]
The {\it distortion}
of $F$ is defined as $\lip(F)\cdot\lip(F^{-1}|_{F(X)})$, where
$\lip(\cdot)$ denotes the Lipschitz constant.

 For
$p\in[1,\infty]$, the {\it
$L_p$-distortion} $c_p(X)$ is defined as the infimum of distortions of all bi-Lipschitz embeddings of $X$ into any space $L_p(\Omega,\Sigma,\mu)$.
\end{definition}

\begin{definition}\label{D:Coarse}
A map $f:(X,d_X)\to (Y,d_Y)$ between two metric spaces is called a
\emph{coarse embedding} if there exist non-decreasing functions
$\rho_1,\rho_2:[0,\infty)\to[0,\infty)$ such that
$\lim_{t\to\infty}\rho_1(t)=\infty$ and
$$\forall u,v\in X~ \rho_1(d_X(u,v))\le
d_Y(f(u),f(v))\le\rho_2(d_X(u,v)).$$
 (Observe that this
condition implies that $\rho_2$ has finite values, but that it does not imply that $f$ is injective.)
\end{definition}

\begin{definition}\label{D:UM} A metric space $(Y,d_Y)$ is called {\it
ultrametric} if  for
any $u,v,w\in Y$
\[d_Y(u,w)\le\max\{d_Y(u,v),d_Y(v,w)\}.\]
\end{definition}

\begin{theorem}[Dadarlat,  Guentner {\cite[Corollary 4.5]{DG07}}] If a metric space
$X$ is a finite union of subsets each admitting  a coarse embedding into
a Hilbert space, then $X$ also admits a coarse embedding into a
Hilbert space.\end{theorem}

\begin{theorem}[Mendel, Naor {\cite[Theorem 1.4]{MN13}}]\label{T:MN13} Let a metric space $(X,d)$ be a union of its metric subspaces $A$ and $B$.
Assume that $A$ and $B$ embed into, possibly different,
ultrametric spaces with distortions $D_A$ and $D_B$, respectively.
Then the metric space $X$ embeds into an ultrametric space with
distortion at most $(D_A+2)(D_B+2)-2$.
\end{theorem}

\begin{theorem}[K.~Makarychev, Y.~Makarychev \cite{MM16}]\label{T:MM}
Suppose that a metric space $(X,d)$ is the union of
two metric subspaces  $A$ and $B$ that embed into $\ell_2^a$ and
$\ell_2^b$ (where $a$ and $b$ may be finite or infinite) with distortions $D_A$ and $D_B$, respectively. Then
$X$ embeds into $\ell_2^{a+b+1}$ with distortion
$D\le 7D_AD_B + 2(D_A+D_B).$

\noindent
If $D_A=D_B = 1$, then $X$ embeds into $\ell_2^{a+b+1}$ with
distortion at most $8.93$.
\end{theorem}

\begin{remark}
We note that there is an extensive literature on the property of $L_1$-embedda\-bi\-lity within the  ``local-global'' theme.
For example, Arora, Lov\'asz, Newman, Rabani, Rabinovich and Vempala \cite{ALNRRV12} asked
what is the least distortion with which one can embed the metric space $X$ into $L_1$, given that
 every subset of $X$ of cardinality $k$ is  embeddable into $L_1$ with distortion at most $D$.
 An 
 answer to this question was given by  Charikar, K.~Makarychev, and Y. Makarychev \cite{CMM10}, who proved, among other results, that if even if a small fraction $\alpha$ (say 1\%) of all subsets of size $k$ of a metric space $X$, with $|X|=n$, embeds into
$\ell_p$ with distortion at most $D$, then  the entire space $X$ embeds into $\ell_p$, $1\le p<\infty$,  with distortion
at most
$D \cdot O(\log(n/k) + \log \log(1/\alpha) + \log p)$. In particular, if $k$ is proportional to $n$, then one obtains a bounded distortion embedding of $X$ into $\ell_p$.

On the other hand, 
 there exist   absolute constants
$a,b,A,B>0$  such that for every $N\in\bbN$ there exists an $N$-point metric space $X$ such that
every subset of $X$ of size at most  $A(\exp(\log \log N)^a)$
 embeds isometrically into $L_1$, but every embedding of $X$ into $L_1$ requires distortion at least $B(\log \log N)^b$, see  \cite{RS09} and  the expositions in
 \cite[Section~1.3]{KV15} and \cite[Section~4.4]{Tre12}.
\end{remark}

In connection with Theorem \ref{T:MM} it became   natural to
investigate the problem whether analogous results are valid  for metric spaces embeddable into $L_p$, when $p\ne 2$, explicitly stated e.g. in \cite[Remark
4.2]{MN13}, \cite[Question 5]{MM16}, \cite[Open Problem
9.6]{Nao15}, and \cite[Remark 18]{NR17}.

\begin{problem}\label{P:Lp}
Suppose $(X, d)$ is a metric space and $X = A\cup B$, with $c_p(A)$ and $c_p(B)$ finite. Does this imply that $c_p(X)$ is finite? Can  $c_p(X)$ be bounded from above only in terms of $c_p(A)$ and $c_p(B)$?
\end{problem}

It is easy to see that the answer is positive for $p=\infty$.
Theorem \ref{T:MM} states that the answer is positive for $p=2$.
K.~Makarychev and  Y.~Makarychev  \cite[Question 5]{MM16}
conjectured that the answer is negative for every $p\in[1,\infty]$
except $2$ and $\infty$.
\medskip

Problem~\ref{P:Lp} is particularly interesting in the case of
$p=1$. In this case, in addition to the Makarychev-Makarychev conjecture of the negative answer to Problem~\ref{P:Lp}, since 2015 in the literature  there were conjectures that a construction known as a twisted union of hypercubes might be a possible method of constructing a family of counterexamples.

\begin{problem}[{\rm Naor \cite[Open problem 3.3]{Nao15}, Naor, Rabani \cite[Remark~18]{NR17}}]\label{P:TwistCube}
 Must any embedding of a  twisted union of hypercubes described in Examples~\ref{E:ConNR&N} and \ref{E:GenNR&N} below into $L_1$  incur a bi-Lipschitz distortion that tends to $\infty$ as the size of the hypercube tends to $\infty$?

\noindent That is, does a twisted union of hypercubes
  give a
negative answer to Problem~\ref{P:Lp} in the case $p=1$?
\end{problem}

The idea of the construction of a twisted union of metric spaces
can be traced back to \cite{Lin64} and has been used in
\cite{JLS86} and \cite{NR17} to provide examples that demonstrate that for $\al\in(1/2,1]$, the $\al$-extension constants from $\ell_\infty$ to $\ell_2$ are not bounded. 
Variants of
this construction were also used in \cite{Lan99, CKR04}.


The most general idea is explained in \cite[Remark 19]{NR17}, and
is as follows: Suppose that $(X, d_X)$ and $(Y, d_Y)$ are finite
metric spaces   with $X$ and $Y$ disjoint as sets. Given mappings
$\s:X\to Y$ and $r:X\to (0,\infty)$, we define the weighted graph
structure on $X\cup Y$ by defining the following weighted edges:
If $x_1, x_2 \in X$ then  $x_1$ and $x_2$ are joined by an edge of
weight $d_X(x_1, x_2)$; if $y_1, y_2 \in Y$ then  $y_1$ and $y_2$
are joined by an edge of weight $d_Y(y_1, y_2)$. Also, for every
$x \in X$, the elements $x$ and $\s(x)$ are joined by an edge of
weight $r(x)$. We endow $X\cup Y$ with the shortest-path metric
induced by this weighted graph.

Naor and Rabani point out that all metric spaces that they
construct in \cite{NR17} to exhibit a maximal unbounded growth of
certain extension constants can be described as subsets of this
general construction, and they indicate that usefulness of this
construction is probably yet to be fully explored.

One of the main goals of the present paper is to start an exploration of general twisted unions of metric spaces. Under certain natural restrictions on the metrics of the original spaces and the function $r:M\to (0,\infty)$ we give detailed formulas for the induced metric of the twisted unions,
see Section~\ref{S:Definition}.

We then study the embedding properties of twisted unions. Somewhat to our surprise, we answer Problem~\ref{P:TwistCube} negatively. In fact we prove a number of results (see Theorems~\ref{fixedr-omega}, \ref{T:rLipschitz}, \ref{T:Basic}, and Corollaries~\ref{fixedr-omega-1}, \ref{fixedr})
establishing that all twisted unions of $L_1$-embeddable spaces that satisfy certain natural mild assumptions, also embed into $L_1$ with distortions bounded by constants that depend only on certain general parameters and not on the metrics themselves or the size of the sets.

While we do not obtain the ultimate answer to Problem~\ref{P:Lp}, we provide a fairly large new class of metric spaces for which the answer to Problem~\ref{P:Lp} is affirmative. While it is interesting to search for a counterexample to Problem~\ref{P:Lp} (which is widely believed to exist), identifying classes of spaces for which the answer to Problem~\ref{P:Lp} is affirmative may have a greater potential for future applications. For this reason we investigate here what  conditions on twisted unions assure their $L_1$-embeddability, see
Sections~\ref{S:twistedemb} and \ref{S:gentwistemb}.

\remove{
Our first main goal is to solve in the negative
Problem~\ref{P:TwistCube} for the twisted union of hypercubes
constructions suggested in \cite{Nao15} and \cite{NR17}. We do
this in somewhat more general cases. We relate our results with
the mentioned constructions at the beginning of Section
\ref{S:Basic} for the construction of \cite[Open problem
3.3]{Nao15} and in Corollary \ref{fixedr-omega-1} for the
construction of \cite[Remark~18]{NR17}. So the twisted unions of
hypercubes constructions suggested in \cite{Nao15} and \cite{NR17}
do embed into $L_1$ with uniformly bounded distortions and thus do
not provide a counterexample  to Problem~\ref{P:Lp} for $p=1$.}

In Section~\ref{S:lowerbound}, we   show that the lower bound on
distortion of the embedding of a union of metric spaces  found by
K.~Makarychev and  Yu.~Makarychev \cite[Theorem 1.2 and Section
3]{MM16} for the Hilbert space, is also valid for all $L_p$ with
$1\le p <\infty$, and for many other Banach and metric spaces. Our
proof uses the theory of stable metric spaces (see
Definition~\ref{D:Stable}) and our examples are infinite metric
spaces. For spaces whose stability is known our proof is very
simple (see Example~\ref{E:UnstUnion}).

\section{Preliminary facts and notation}

We use the standard terminology of the theories of Banach Spaces
and Metric Embeddings,  see \cite{BL00, Ost13}.

Suppose that  $K\ge1$,  $M$ is a set, and $f:M\times M\to[0,\infty)$ is an arbitrary function that is not necessarily a metric on $M$,  such that there exists a map $\Psi:M\to L_1$ such that  for all $x,y\in M$, we have
\[f(x,y)\le\|\Psi(x)-\Psi(y)\|\le Kf(x,y).\]
In this situation, even if the function $f$ is not a metric (we even allow the function $f$ to equal to 0 on an arbirtary subset of $M\times M$), with a slight abuse of notation, we will say that $(M,f)$ embeds in $L_1$ and   write  $c_1(M,f)\le K<\infty$.

Please see Remark~\ref{R:semimetrics} for a brief discussion of functions $f$ that can satisfy this condition.

We will  need
the following two results of Mendel and Naor \cite{MN15}.

\begin{theorem} \label{T:trunc}
For every $\la>0$, $L_1$ with the truncated metric
$$\varrho(x,y)=\min\{\la,\|x-y\|_1\}$$
embeds into $L_1$ with distortion not exceeding $e/(e-1)$.
\end{theorem}

This is a powerful result which, using  the
theory of concave functions of Brudnyi and Krugljak (see
\cite[Section 3.2]{BK91} and \cite[Remark 5.4]{MN04}),  implies:

\begin{corollary}\label{C:concave}{\rm \cite{MN15}}
There exists a universal constant $\Delta\le (2\sqrt{2}+3)e/(e-1)<10$,
such that if $\om:[0,\infty)\to[0,\infty)$ is any concave
non-decreasing function with $\om(0)=0$ and $\om(t)>0$ for $t>0$,
then the metric space $(L_1,\om(\|x-y\|_1))$ embeds into $L_1$
with distortion at most $\Delta$.
\end{corollary}

We will also use the following immediate corrolary of Theorem~\ref{T:trunc}

\begin{corollary}\label{C:trunc}
Let $M$ be a set, $D\ge 1$, and $f:M\times M\to[0,\infty)$ be a function   such that there exists a map $\phi:M\to L_1$ such that  for all $x,y\in M$, we have
\[f(x,y)\le\|\phi(x)-\phi(y)\|\le Df(x,y).\]
Then, for any constant $\la>0$, $c_1(M, \min\{f(x,y),\la\})\le eD/(e-1)$.
\end{corollary}

\begin{proof}
By Theorem~\ref{T:trunc}, there exists a map $T:L_1\to L_1$ such that for all  $u,v\in L_1$
\[ \min\{\|u-v\|, \la\}\le \|Tu-Tv\|\le \frac{e}{e-1}\min\{\|u-v\|, \la\}.\]

Thus for all $x,y\in M$
\begin{equation*}
\begin{split}
 \|T\phi(x)-T\phi(y)\|&\le \frac{e}{e-1}\min\{\|\phi(x)-\phi(y)\|, \la\}\le \frac{e}{e-1}\min\{Df(x,y), \la\}\\
 &\le \frac{eD}{e-1}\min\{f(x,y), \la\}
 \end{split}
\end{equation*}
and
\begin{equation*}
\begin{split}
 \|T\phi(x)-T\phi(y)\|\ge \min\{\|\phi(x)-\phi(y)\|, \la\}
 \ge \min\{f(x,y), \la\},
 \end{split}
\end{equation*}
which ends the proof.
\end{proof}

If $G$ and $H$ are real  valued quantities  or functions, we use notation $G\asymp H$ to mean that there exist $0<\alpha\le
\beta<\infty$ such that $\alpha G\le H\le \beta G$. The numbers
$\alpha$ and $\beta$ can depend on the parameters mentioned in the
statements of the results we are proving, but not on elements
$x,y$ of the considered metric space.

\section{Twisted unions of
hypercubes and of general metric spaces}\label{S:Definition}

We start by presenting the concrete examples of twisted unions of
hypercubes from \cite{JLS86},  \cite{Nao15}, and \cite{NR17}.

\begin{example}\cite{JLS86,Nao15,NR17}\label{E:ConNR&N}
Let $\mathbb{F}^n_2=\{0,1\}^n$ be the Hamming cube embedded into
$\ell_1^n$ in a natural way and for $x,y\in \mathbb{F}^n_2$ we use
the notation $\|x-y\|$  for the $\ell_1$-norm. Let  $r\in
(0,\infty)$ and $\al\in(1/2,1]$ be fixed constants.

Define the metric on $\mathbb{F}^n_2\times
\mathbb{F}_2=\{0,1\}^n\times \{0,1\}$ as the shortest path metric when  $\mathbb{F}^n_2\times\mathbb{F}_2$ is considered as a graph with the following edges and weights:
\begin{itemize}
\item for every $x,y\in \mathbb{F}^n_2$ there is an edge with ends $(x,0)$ and $(y,0)$ of weight $\|x-y\|^{\frac{1}{2\al}}$,
\item for every $x,y\in \mathbb{F}^n_2$ there is an edge with ends $(x,1)$ and $(y,1)$ of weight $\frac{\|x-y\|}{r^{2\al-1}}$,
\item for every $x\in \mathbb{F}^n_2$ there is an edge with ends $(x,0)$ and $(x,1)$ of weight $r$.
\end{itemize}
The exact formula for the distance between any two points of $\mathbb{F}^n_2\times\mathbb{F}_2$ in this metric is computed in
\cite[Lemma~15]{NR17}, cf. also \eqref{E:distconcave} and Remark~\ref{R:dx0y1-r-2} below.
\begin{equation}\label{E:distconcave-nr}
  d((x,a),(y,b))=\begin{cases}
   \|x-y\|^{\frac{1}{2\al}}, \  &\text{if\ \ } a=b=0,\\
   \min\{\frac{\|x-y\|}{r^{2\al-1}},2r+\|x-y\|^{\frac{1}{2\al}}\}, \  &\text{if\ \ } a=b=1,\\
    r\ \ &\text{if\ \ } x=y, a\ne b,\\
   r+\min\{\frac{\|x-y\|}{r^{2\al-1}},\|x-y\|^{\frac{1}{2\al}}\}\ \ &\text{if\ \ }  a\ne b.
   \end{cases}
\end{equation}

\end{example}

In \cite{NR17} it was   shown that Example~\ref{E:ConNR&N} is a special case of the following more general example.

\begin{example}\label{E:GenNR&N}\cite[Remark 18]{NR17}
Let $\mathbb{F}^n_2$ be the Hamming cube, $r\in (0,\infty)$ be a constant, and
$\omega_0$ and $\omega_1$ be concave
non-decreasing continuous functions on $[0,\infty)$ vanishing at $0$ and such that for all $t>0$, $\om_0(t)>0$ and $\om_1(t)>0$.
Define the metric on $\mathbb{F}^n_2\times
\mathbb{F}_2$ as the shortest path metric when
 $\mathbb{F}^n_2\times\mathbb{F}_2$ is considered as a graph with the following edges and weights:
\begin{itemize}
\item for every $x,y\in \mathbb{F}^n_2$ there is an edge with ends $(x,0)$ and $(y,0)$ of weight $\om_0(\|x-y\|)$,
\item for every $x,y\in \mathbb{F}^n_2$ there is an edge with ends $(x,1)$ and $(y,1)$ of weight $\om_1(\|x-y\|)$,
\item for every $x\in \mathbb{F}^n_2$ there is an edge with ends $(x,0)$ and $(x,1)$ of weight $r$.
\end{itemize}

Since the functions $\omega_0$ and $\omega_1$ are concave,
it is easy to see that   the above weights imply that the shortest path  metric on $\mathbb{F}^n_2\times
\mathbb{F}_2$ satifies
\begin{equation}\label{E:distconcave}
  d((x,a),(y,b))=\begin{cases}
   \min\{\om_0(\|x-y\|),2r+\om_1(\|x-y\|)\}, \  &\text{if\ \ } a=b=0,\\
   \min\{\om_1(\|x-y\|),2r+\om_0(\|x-y\|)\}, \  &\text{if\ \ } a=b=1,\\
    r\ \ &\text{if\ \ } x=y, a\ne b.
   \end{cases}
\end{equation}
Moreover, by Lemma~\ref{L:dx0y1-rg}(c) below, if $a\ne b$,
$$d((x,a),(y,b))=r+\min\{\om_0(\|x-y\|),\om_1(\|x-y\|)\}.$$

Notice that if we  define  
\begin{equation*}
\begin{split}
\varpi_0(t)&= \min\{\om_0(t),2r+\om_1(t)\}, \\
\varpi_1(t)&= \min\{\om_1(t),2r+\om_0(t)\},
\end{split}
\end{equation*}
 then the functions $\varpi_0, \varpi_1$ are concave,
non-decreasing,  continuous,  vanishing only at $t=0$, and for all $t$
$$|\varpi_0(t)-\varpi_1(t)|\le 2r.$$
It is clear that we could have used functions $\varpi_0(\|x-y\|),\varpi_1(\|x-y\|)$ as weights in place of $\om_0(\|x-y\|),\om_1(\|x-y\|)$, respectively.

Naor and Rabani \cite[Theorem 16]{NR17} proved that the $\al$-snowflake (for the same $\al\in(1/2,1]$ that is used as a parameter in the definition of the space) of the metric space considered in Example~\ref{E:ConNR&N} on the set $\mathbb{F}^n_2\times\{0\}$  embeds isometrically into $\ell_2$, but that for certain choices of the value of $r>0$, the minimum distortion of any embedding of the space $(\mathbb{F}^n_2\times\mathbb{F}_2,d^\al)$ into $\ell_2$ is at least of the order ${\disp n^{\frac{2\al-1}{4\al}}}$, and thus the $\al$-extension constant grows to $\infty$ when the size of the hypercube grows to $\infty$: ${\disp e_m^\al(\ell_\infty,\ell_2) \gtrsim (\log m)^{\frac{2\al-1}{4\al}}}$.

Naor and Rabani \cite[Remark 18]{NR17} also noted that, by a
result from \cite{MN15} (see Corollary~\ref{C:concave}),
 for all $n$, the metric spaces
$(\mathbb{F}^n_2\times\{0\},
\varpi_0(\|x-y\|))$ and $(\mathbb{F}^n_2\times\{1\}, \varpi_1(\|x-y\|))$,
embed into $L_1$ with bounded distortions,
 and they asked whether any embedding of
 $(\mathbb{F}^n_2\times\mathbb{F}_2,d)$
 must incur a bi-Lipschitz distortion that tends to $\infty$ as $n\to \infty$, and commented that this would be the first known example of a metric space that can be partitioned into two subsets, each of which
well-embeds into $\ell_1$ yet the entire space does not, see also \cite[Open Problem
3.3]{Nao15}.
\end{example}

Following the ideas of Examples~\ref{E:ConNR&N} and
\ref{E:GenNR&N}, and of \cite[Remark 19]{NR17} we state the
 definitions of a twisted union  and a generalized twisted union of metric spaces.

\begin{definition}[Twisted Union  and Generalized Twisted Union of Metric Spaces]
\label{D:TwUn}
Let $M$ be a finite set which is endowed with two
metrics, $d_0$ and $d_1$, and let $r$ be a fixed positive number (respectively, let $r$ be a function from $M$ to $(0,\infty)$)   such that for all $x,y \in M$:
\begin{equation}\label{E:Cndr2-c}
|d_0(x,y)-d_1(x,y)|\le 2r,
\end{equation}
or, respectively,
\begin{align}\label{E:Cndr2}
&|d_0(x,y)-d_1(x,y)|\le r(x)+r(y),\\
\label{E:Cndr1}
{\text {\rm and}\ \ \ }&|r(x)-r(y)|\le d_0(x,y)+d_1(x,y).
\end{align}

Define the metric $d$ on
$M\times\mathbb{F}_2=M\times\{0,1\}$ as the shortest path metric when
 $M\times\mathbb{F}_2$ is considered as a graph with the following edges and weights:
\begin{itemize}
\item for every $x,y\in M$ there is an edge with ends $(x,0)$ and $(y,0)$ of weight $d_0(x,y)$,
\item for every $x,y\in M$ there is an edge with ends $(x,1)$ and $(y,1)$ of weight $d_1(x,y)$,
\item for every $x\in M$ there is an edge with ends $(x,0)$ and $(x,1)$ of weight $r$ (resp., $r(x)$).
\end{itemize}

If $r>0$ is a constant, the space $(M\times \mathbb{F}_2,d)$ is called the {\it twisted union of metric spaces $(M,d_0)$ and $(M,d_1)$ with the joining parameter $r$},  or the {\it $r$-twisted union}, for short.

 If $r(\cdot)$ is a  function, $(M\times \mathbb{F}_2,d)$ is called the {\it generalized twisted union  of $(M,d_0)$ and $(M,d_1)$ with the joining function $r:M\to(0,\infty)$}.
\end{definition}

We note that, unlike the situation in Example~\ref{E:GenNR&N}, for a space $M$ with arbitrary metrics
 $d_0, d_1$ that are not necessarily given as concave functions of another underlying metric, it is not clear how to derive an analogue of the formula \eqref{E:distconcave}, that is, a formula for the metric of the twisted union restricted to the sets $M\times \{0\}$ and $M\times \{1\}$. For this reason we impose conditions on relations between the metrics  $d_0, d_1$, and the joining parameter $r$  or the joining function $r:M\to (0,\infty)$, respectively.
The condition \eqref{E:Cndr2-c} (resp., conditions \eqref{E:Cndr2} and \eqref{E:Cndr1})
are necessary and sufficient for the weights of all edges   to be equal to the weighted graph distance between
their ends.
Hence the metric $d$ on the twisted union (resp., the generalized twisted union) $M\times \mathbb{F}_2$ satisfies
\begin{equation}\label{E:defd}
  d((x,a),(y,b))=\begin{cases}
   d_0(x,y) \ \ &\text{if\ \ } a=b=0,\\
   d_1(x,y) \ \ &\text{if\ \ } a=b=1,\\
   r \ \ {(\text{\rm  resp., }}    r(x) ) \ &\text{if\ \ } x=y, a\ne b.
   \end{cases}
\end{equation}

Computation of the formula for $d((x,a),(y,b))$ when $x\ne y$ and $a\ne b$ is more delicate and we did it only when certain additional conditions on $d_0, d_1$ and $r$ (resp. $r(\cdot)$) are satisfied.


We prove, under two different natural assumptions that are independent of each other, that for all
$x,y\in M$,
\begin{equation}\label{E:dx0y1gen}
\begin{split}
  d((x,0),(y,1))
  &\asymp \min\{d_0(x,y), d_1(x,y)\}+r(x).
  \end{split}
\end{equation}

By Lemma~\ref{L:dx0y1-rg} below, the formula \eqref{E:dx0y1gen} is valid if
 the joining function $r(\cdot)$ is Lipschitz with respect to both metrics $d_0$ and $d_1$, not just with respect to the sum of these metrics as required in \eqref{E:Cndr1}, that is, if there exists a constant $L>0$, such that for all $x,y\in M$,
\begin{equation}\label{E:rLip}
 |r(x)-r(y)|\le L\min\{d_0(x,y), d_1(x,y)\}.
\end{equation}

In particular,  \eqref{E:dx0y1gen}  is satisfied when the joining function $r(\cdot)$ is constant, since in this case \eqref{E:rLip} holds with $L=0$.


 By Lemma~\ref{SL:dBtwCbs} below, the formula \eqref{E:dx0y1gen} is valid if  there exists a constant $C>0$ such that for all $x,y\in M$,
\begin{equation}\label{E:d1lecd0}
 d_1(x,y)\le Cd_0(x,y).
\end{equation}

Condition \eqref{E:d1lecd0} is satisfied, for example, in Example~\ref{E:ConNR&N}, since there, independent of the value of $r>0$ or $\al\in(1/2,1]$, for all $x,y\in M$,  we have
\begin{equation}\label{E:Ex3.1C}
 \varpi_1(\|x-y\|)=\min\Big\{\frac{\|x-y\|}{r^{2\al-1}},2r+\|x-y\|^{\frac{1}{2\al}}\Big\}\le 2 \|x-y\|^{\frac{1}{2\al}}=2\varpi_0(\|x-y\|).
 \end{equation}

It is clear that conditions \eqref{E:rLip} and \eqref{E:d1lecd0} are independent of each other.

We note that   \eqref{E:rLip} implies that
\begin{equation*}\label{E:danebrLip}
\begin{split}
  d((x,0),(y,1)) &\asymp \min\{d_0(x,y), d_1(x,y)\}+r(x)\\
  &\asymp \min\{d_0(x,y), d_1(x,y)\}+\max\{r(x), r(y)\}.
  \end{split}
\end{equation*}

On the other hand, \eqref{E:d1lecd0} implies that $\min\{d_0(x,y), d_1(x,y)\}\asymp d_1(x,y)$.
Thus, if \eqref{E:d1lecd0} holds, then


\begin{equation}\label{E:danebdCd0}
  d((x,0),(y,1)) \asymp \min\{d_0(x,y), d_1(x,y)\}+r(x)\asymp d_1(x,y)+r(x).
\end{equation}
 However in   \eqref{E:danebdCd0} it is not possible to replace $r(x)$ by either $r(y)$,
$\max\{r(x), r(y)\}$, or $\min\{r(x), r(y)\}$, since \eqref{E:d1lecd0} does not imply that the function $r(\cdot)$ is Lipschitz with respect to the metric  $d_1$.

We now prove the lemmas mentioned above. The common assumption of
Lemmas~\ref{L:dx0y1-rg} and  \ref{SL:dBtwCbs} is:
\begin{equation}
  \tag{$*$}\label{eq:A}
  \parbox{\dimexpr\linewidth-4em}{%
    \strut
    Let $M$ be a metric space endowed with two metrics $d_0$ and
$d_1$,   and let $r:M\to (0,\infty)$ be a function such that \eqref{E:Cndr2} and \eqref{E:Cndr1}
are satisfied, and let $(M\times \mathbb{F}_2,d)$ be the  generalized twisted union
 of $(M,d_0)$  and $(M,d_1)$ with the joining function $r(\cdot)$.
    \strut
  }%
\end{equation}

\begin{lemma}\label{L:dx0y1-rg}
 Suppose $(*)$ and that the joining function $r(\cdot)$ is Lipschitz with respect to both metrics $d_0$ and
$d_1$,  that is, suppose that there exists a constant $L>0$, such that for all $x,y\in M$
\begin{equation}\label{E:rhLipschitz}\tag{3.8}
 |r(x)-r(y)|\le L\min\{d_0(x,y), d_1(x,y)\}.
\end{equation}

Denote $\min\{d_0(x,y), d_1(x,y)\}$ by $h(x,y)$.

 \begin{enumerate}[{\rm (a)}]
\item Then for all $x,y\in M$
\begin{equation}\label{E:dx0y1-rg}
 \frac1{A}\Big(h(x,y)+\max\{r(x), r(y)\}\Big)\le d((x,0),(y,1))\le
  h(x,y)+\max\{r(x), r(y)\},
\end{equation}
where $A=\max\{2L+1,3\}$.

\item If the joining function is constant with $r(x)=r>0$, for all
$x\in M$, then for all $x,y\in M$
\begin{equation}\label{E:dx0y1-r}
 \frac1{3}\Big(r+ h(x,y)\Big)\le d((x,0),(y,1))\le
  r+ h(x,y).
\end{equation}
\item If $h(x,y)=\min\{d_0(x,y), d_1(x,y)\}$ is a metric on $M$
and the joining function is constant with $r(x)=r>0$, for all
$x\in M$, then for all $x,y\in M$
\begin{equation}\lb{E:dx0y1-r-2}
d((x,0),(y,1))=
  r+ h(x,y).
  \end{equation}
\end{enumerate}
\end{lemma}

\begin{remark}\lb{R:dx0y1-r-2}
We included \eqref{E:dx0y1-r-2} above, because we are particularly interested in spaces  described in Example~\ref{E:GenNR&N},  when
$M$ is a subset of $L_1$ and $d_i(x,y)=\varpi_i(\|x-y\|)$ for $i=0,1$, where functions $\varpi_0, \varpi_1$ are concave,
non-decreasing,  continuous, and vanishing only at $t=0$, see Corollary~\ref{fixedr-omega-1}. In this situation $\min\{d_0(x,y), d_1(x,y)\}$ is a metric on $M$.

We note that if $\min\{d_0(x,y), d_1(x,y)\}$ is a metric on $M$, then one can obtain slightly better constants also in \eqref{E:dx0y1-rg} above.
\end{remark}

\begin{proof}
We fix $x,y\in M$.
The condition \eqref{E:Cndr2} implies that on a shortest path from
$(x,0)$ to $(y,1)$ we may avoid moving from $M\times \{0\}$ to
$M\times \{1\}$ more than once. Thus
\[d((x,0),(y,1))=\inf_{z\in M}\left(d_0(x,z)+d_1(z,y)+r(z)\right).\]
If $h(x,y)=d_1(x,y)$, we pick $z=x$, otherwise we pick $z=y$, so that  we get
\begin{equation*}\lb{E:dx0y1-r-2ag}
d((x,0),(y,1))\le h(x,y)+ \max\{r(x), r(y)\}.
\end{equation*}

On the other hand, by \eqref{E:Cndr2},   for all $u,v\in M$ we have
$d_0(u,v)\le d_1(u,v)+r(u)+r(v)$. Thus, for every $x,y,z\in M$ we have
\begin{equation}\lb{d0d1estg}
\begin{split}
h(x,y)&\le d_0(x,y)\le d_0(x,z)+d_0(z,y)\le
d_0(x,z)+d_1(z,y)+r(z)+ r(y)\\&\le d_0(x,z)+d_1(z,y)+r(z)+\max\{r(x), r(y)\}.
\end{split}
\end{equation}

Hence, for every $x,y,z\in M$ and $T>1$ we have
\[\begin{split}
d_0&(x,z)+d_1(z,y)+r(z)
-\frac1T\Big(h(x,y)+ \max\{r(x), r(y)\}\Big)\\
&\stackrel{\eqref{d0d1estg}}{\ge} d_0(x,z)+d_1(z,y)+
r(z)-
\frac1T\Big(d_0(x,z)+d_1(z,y)+r(z)+2\max\{r(x), r(y)\}\Big)\\
&\ge
\Big(1-\frac1T\Big)\Big(d_0(x,z)+d_1(z,y)\Big)+\Big(1-\frac3T\Big)r(z) -\frac2T\Big|r(z)-\max\{r(x), r(y)\}\Big|\\
&\stackrel{\eqref{E:rhLipschitz}}{\ge}
\Big(1-\frac1T\Big)\Big(h(x,z)+h(z,y)\Big)+\Big(1-\frac3T\Big)r(z) -\frac{2L}T\max\{h(x,z), h(z,y)\}\\
&\ge
\Big(1-\frac{2L+1}T\Big)\max\{h(x,z), h(z,y)\}+\Big(1-\frac3T\Big)r(z)
\end{split}\]
If $T=\max\{2L+1,3\}$, the ultimate quantity is nonnegative, which proves
\eqref{E:dx0y1-rg}.

  Formula \eqref{E:dx0y1-r} immediately follows from \eqref{E:dx0y1-rg}, since when the function $r(\cdot)$ is constant then it is Lipschitz with the Lipschitz constant $L=0$.

If $h(x,y)=\min\{d_0(x,y), d_1(x,y)\}$ is a metric on $M$, then
for all $z\in M$ we have
  \[\begin{split}
d_0(x,z)+d_1(z,y)+r
&\ge\min\{d_0(x,z), d_1(x,z)\} +\min\{d_0(z,y), d_1(z,y)\}+r\\
&\ge
 \min\{d_0(x,y), d_1(x,y)\}+r,
\end{split}\]
which proves \eqref{E:dx0y1-r-2}.
\end{proof}

\begin{lemma}\label{SL:dBtwCbs}  Suppose $(*)$. Then conditions \eqref{E:Cndr2}, \eqref{E:Cndr1}, and
\begin{equation}\label{E:d1lecd0-22}\tag{3.9}
\forall x,y\in M\quad d_1(x,y)\le Cd_0(x,y).
\end{equation}
imply that
\begin{equation}\label{E:dx0y1}
 \frac1{2C+1}\Big(d_1(x,y)+r(x)\Big)\le d((x,0),(y,1))\le d_1(x,y)+r(x).
\end{equation}

\noindent
Moreover, if $r(x)=r>0$ for all $x\in M$, then \eqref{E:Cndr2},
\eqref{E:Cndr1}, and \eqref{E:d1lecd0-22} imply that
\begin{equation}\label{E:dx0y1-2}
 \frac1{C+1}\Big(d_1(x,y)+r\Big)\le d((x,0),(y,1))\le d_1(x,y)+r.
\end{equation}
\end{lemma}

\begin{proof} The condition \eqref{E:Cndr2} implies that on
a shortest path from $(x,0)$ to $(y,1)$ we may avoid moving from
$M\times \{0\}$ to $M\times \{1\}$ more than once. Thus
\[d((x,0),(y,1))=\inf_{z\in M}\left(d_0(x,z)+d_1(z,y)+r(z)\right).\]
If we pick $z=x$, we get

\[d((x,0),(y,1))\le d_1(x,y)+r(x).\]

On the other hand, for every $z\in M$ and $T>1$ we have
\[\begin{split}
d_0(x,z)&+d_1(z,y)+r(z)
-\frac1T\Big(d_1(x,y)+r(x)\Big)\\
&\ge d_0(x,z)+d_1(z,y)-\frac1Td_1(x,y)-
\frac1T\Big(r(x)-r(z)\Big)\\
&\stackrel{\eqref{E:Cndr1}}{\ge}
d_0(x,z)+d_1(z,y)-\frac1Td_1(x,y)-
\frac1T\Big(d_0(x,z)+d_1(x,z)\Big)\\
&\ge\Big(1-\frac1T\Big)d_0(x,z)+d_1(z,y)-\frac1Td_1(x,y)-
\frac1Td_1(x,z)\\
&\ge\Big(1-\frac1T\Big)d_0(x,z)-\frac1Td_1(x,z)-
\frac1Td_1(x,z)\\
&\stackrel{\eqref{E:d1lecd0-22}}{\ge}\frac1{C}\Big(1-\frac1T\Big)d_1(x,z)-\frac2Td_1(x,z)\\
&\ge\frac{T-1-2C}{CT}d_1(x,z).
\end{split}\]
The ultimate quantity is nonnegative if $T=2C+1$. This proves
\eqref{E:dx0y1}. The proof of \eqref{E:dx0y1-2} is essentially the
same if we notice that $r(x)-r(z)=0$.
\end{proof}

\section{On $L_1$-embeddability of twisted unions}\lb{S:twistedemb}

\remove{This section aims to prove upper bounds for $L_1$-distortions of
twisted unions of $L_1$-embeddable metric spaces. We demonstrate
such results in a variety of regimes. The assumptions in some of
them are somewhat complicated. We organized our presentation in
this way to achieve the following goals: readers can see (1) how
the argument looks like in the simplest nontrivial case; (2) a
variety of assumptions under which the $L_1$-embeddability still
holds.}

In this section we study $L_1$-embeddability of twisted unions with uniform weight on all edges joining points from distinct original metric spaces.  We start from a general result and, as an application, we obtain that all metric spaces from Examples~\ref{E:ConNR&N} and \ref{E:GenNR&N}  are $L_1$-embeddable, see Corollary~\ref{fixedr-omega-1}. Another corollary of Theorem~\ref{fixedr-omega} is presented in Section~\ref{S:gentwistemb}, see Corollary~\ref{fixedr}.

\begin{theorem}\label{fixedr-omega}
Let  $r>0$ and $M$ be a metric space endowed with two metrics $d_0$ and
$d_1$,  such that \eqref{E:Cndr2-c}
 is satisfied, and for $i=0,1$, the space $(M,d_i)$ embeds into
$L_1$ with distortion at most $D_i$. Suppose also that there exist
a constant $K\ge 1$ and a map $\psi:M\to L_1$ such that for all
$x,y\in M$
\begin{equation}\label{minemb}
\min\{d_0(x,y),d_1(x,y)\}\le \|\psi(x)-\psi(y)\|\le
K\min\{d_0(x,y),d_1(x,y)\}.
\end{equation}

 Then the $r$-twisted union
$(M\times \mathbb{F}_2,d)$ of $(M,d_0)$  and $(M,d_1)$ 
 embeds into $L_1$ with distortion
bounded by a  constant $D$ that depends only on $D_0, D_1$ and $K$.
\end{theorem}

\begin{remark}\lb{R:semimetrics}
In general the minimum of two metrics does not need to be a metric.
However if \eqref{minemb} is satisfied, then there exists a metric $\gamma$ on $M$, and a constant $\beta\in(0,1]$ such that $2^{1/\be}=2K$, and for  all $x,y\in M$,
\[\frac14 (\min\{d_0(x,y),d_1(x,y)\})^\beta\le \gamma(x,y)\le
(\min\{d_0(x,y),d_1(x,y)\})^\beta.\]

This follows from a routine adjustment of  a result of Kalton,  Peck,  and  Roberts \cite[Theorem~1.2]{KPR84}, who studied properties of generalizations of F-norms that instead of the usual triangle inequality satisfy the ultimate inequality in  
\begin{equation}\label{V}
h(x,y)\le K\big(h(x,z)+h(z,y)\big)\le 2K\max\big\{h(x,z),h(z,y)\big\}.
\end{equation}

Clearly, if $h(x,y)\DEF\min\{d_0(x,y),d_1(x,y)\}$, then \eqref{minemb} implies \eqref{V} and  $h$ is separating ($h(x,y)=0$ iff $x=y$) and symmetric ($h(x,y)=h(y,x)$ for all $x,y\in M$).

We note that such functions were studied already by Fr\'echet \cite{Fre1906,Fre1913}, who called any symmetric, separating function $h:M\times M\to[0,\infty)$   satisfying  a  condition sligthy weaker than \eqref{V} a {\it voisinage}.  Chittenden \cite{Chi1917} proved that any space with a voisinage is homeomorphic to a metric space. For modern theory of similar types of spaces see the monograph of Kalton,  Peck,  and  Roberts \cite{KPR84}.
\end{remark}

\begin{proof}
Define for all $x,y\in M$ and $i=0,1$
\begin{equation}\label{E:d2}
h(x,y)=\min\{d_0(x,y),d_1(x,y)\}
\end{equation}
\begin{equation}\label{E:defrho}
\rho_i(x,y)=\min\{d_i(x,y),2r\}.
\end{equation}

Then, for $i=0,1$, $x,y\in M$ we have
\begin{equation}\label{E:rho}
d_i(x,y)\le \rho_i(x,y)+h(x,y)\le 2d_i(x,y).
\end{equation}

Indeed, by \eqref{E:d2} and \eqref{E:defrho}, the rightmost
inequality is clear. If for some $x,y\in M$, either
$\rho_i(x,y)=d_i(x,y)$ or $h(x,y)=d_i(x,y)$ then the leftmost
inequality also holds. The remaining case is when $\rho_i(x,y)=2r$
and $h(x,y)=d_{i+1}(x,y)$, where $i+1=1$ if $i=0$, and $i+1=0$
if $i=1$. In this case, by \eqref{E:Cndr2-c},
\[d_i(x,y)\le 2r+d_{i+1}(x,y),\]
which proves \eqref{E:rho}.

Since $c_1(M,d_i)\le D_i$, for $i=0,1$, by Corollary~\ref{C:trunc}, there exist maps $\vf_i:M\to L_1$   such that for all $x,y\in M$
\[\min\{d_i(x,y), 2r\}\le \|\vf_i(x)-\vf_i(y)\|\le \frac{eD_i}{e-1}\min\{d_i(x,y), 2r\}.\]

Let $m_0\in M$ be a fixed element of $M$.

We define an embedding of $(M\times \mathbb{F}_2,d)$ into
$L_1\oplus_1L_1\oplus_1L_1\oplus_1\mathbb{R}$ by
\begin{equation*}
\begin{split}
G(x,0)&=(\vf_0(x),\vf_1(m_0), \psi(x), r)\\
G(x,1)&=(\vf_0(m_0),\vf_1(x), \psi(x), 0).
\end{split}
\end{equation*}

We have
\[\begin{split}\|G(x,0)-G(y,0)\|&=\|\vf_0(x)-\vf_0(y)\|+\|\vf_1(m_0)-\vf_1(m_0)\|
+\|\psi(x)-\psi(y)\|+|r-r|\\&\stackrel{\eqref{minemb}}{\asymp}
\rho_0(x,y)+h(x,y)\\
&\stackrel{\eqref{E:rho}}{\asymp}d_0(x,y).
\end{split}\]

\[\begin{split}\|G(x,1)-G(y,1)\|&=\|\vf_0(m_0)-\vf_0(m_0)\|+\|\vf_1(x)-\vf_1(y)\|
+\|\psi(x)-\psi(y)\|+ |0| \\
&\asymp
\rho_1(x,y)+h(x,y)\stackrel{\eqref{E:rho}}{\asymp}d_1(x,y).
\end{split}\]

By Lemma~\ref{L:dx0y1-rg}(b) we have
\[d((x,0),(y,1))\asymp r+h(x,y).\]

We need to compare this with

\[\begin{split}\|G(x,0)-G(y,1)\|&=\|\vf_0(x)-\vf_0(m_0)\|+\|\vf_1(m_0)-\vf_1(y)\|
+\|\psi(x)-\psi(y)\|+|r-0|\\
&\asymp
\rho_0(x,m_0)+\rho_1(m_0,y)+h(x,y)+r\\&\stackrel{\eqref{E:defrho}}{\asymp}
r+h(x,y). \hfill \qedhere\end{split}\]
\end{proof}

As an application  of Theorem~\ref{fixedr-omega}, we obtain the $L_1$-embeddability of twisted unions that are described in Example~\ref{E:GenNR&N} and includes, in particular,  the twisted union   described in Example~\ref{E:ConNR&N}. Thus
Corollary~\ref{fixedr-omega-1} answers Problem~\ref{P:TwistCube} in the negative (see also
Corollary~\ref{fixedr} below, for different proofs that the   twisted union   from Example~\ref{E:ConNR&N} is $L_1$-embeddable).

\begin{corollary}\label{fixedr-omega-1}
Let $M$ be subset of $L_1$,  $r>0$ be a constant,
$\varpi_0$ and $\varpi_1$ be concave non-decreasing continuous
functions on $[0,\infty)$ vanishing only at $0$, such that for all
$t>0$,
$|\varpi_0(t)-\varpi_1(t)|\le 2r,$ that is
\eqref{E:Cndr2-c}
 is satisfied for the metrics $d_i(x,y)\DEF  \varpi_i(\|x-y\|)$, $i=0,1$.

Then   the $r$-twisted union  $(M\times \mathbb{F}_2,d)$ of
$(M,\varpi_0(\|x-y\|))$ and $(M,\varpi_1(\|x-y\|))$    embeds into
$L_1$ with distortion $D$ bounded by an absolute constant
$(D\le1+2.776\Delta <26.6)$.
\end{corollary}

\begin{proof}
Define for all $t\ge 0$
\begin{equation*}
\varpi_2(t)=\min\{\varpi_0(t),\varpi_1(t)\}.
\end{equation*}

Then $\varpi_2$ is a concave non-decreasing continuous function on
$[0,\infty)$ vanishing  only at $0$.
Thus, by Corollary~\ref{C:concave}, for $i=0,1,2$, $j=0,1$, the spaces
$(M,\varpi_i(\|x-y\|))$ and $(M,\rho_j)$ embed into $L_1$ with distortions bounded
above by $\Delta\le (2\sqrt{2}+3)e/(e-1)$. Thus \eqref{minemb} is satisfied and the corollary
follows by Theorem~\ref{fixedr-omega}.

To obtain the estimate in the parenthesis we use \eqref{E:dx0y1-r-2} and the maps $\vf_j:M\to L_1$ for  $j=0,1$, such that for all $x,y\in M$
\[0.694\cdot \rho_j(x,y)\le \|\vf_j(x)-\vf_j(y)\|\le 0.694\Delta  \rho_j(x,y).\]
The computation  is easy but a little tedious. We leave it to an interested reader.
\end{proof}

\section{On $L_1$-embeddability of generalized twisted unions}\lb{S:gentwistemb}

In this section we present two general results (Theorems~\ref{T:rLipschitz} and \ref{T:Basic})
on $L_1$-embeddability of generalized twisted unions which satisfy different natural restrictions, described in Lemmas~\ref{L:dx0y1-rg} and \ref{SL:dBtwCbs}, for which we obtained an equivalent formula for the twisted union distance between every pair of points of the union. As an application we obtain another proof, different from the one in Section~\ref{S:twistedemb}, that the twisted union described in Example~\ref{E:ConNR&N} embeds into $L_1$ with bounded distortion, see Corollary~\ref{fixedr}.

\begin{theorem}\label{T:rLipschitz}
Let $M$ be a metric space endowed with two metrics $d_0$ and
$d_1$,   and let $r:M\to (0,\infty)$ be a function such that \eqref{E:Cndr2} and \eqref{E:Cndr1}
are satisfied.

 Suppose   that the  function $r(\cdot)$ is Lipschitz with respect to both metrics $d_0$ and
$d_1$,  that is, suppose that there exists a constant $L>0$, such that for all $x,y\in M$
\begin{equation}\label{E:rhLipschitz-2}
 |r(x)-r(y)|\le L\min\{d_0(x,y), d_1(x,y)\}.
\end{equation}
Denote
\[h(x,y)\DEF \min\{d_0(x,y), d_1(x,y)\},\]
and, for $i=0,1$,
\[g_i(x,y)\DEF \min\{d_i(x,y), r(x)+r(y)\}.\]

If there exist constants $C_0,C_1,C_2$ such that for $i=0,1$,
$$c_1(M,g_i)\le C_i\ \  \text{and }\ \  c_1(M,h)\le C_2,$$
then  the  generalized twisted union $(M\times \mathbb{F}_2,d)$
 of $(M,d_0)$  and $(M,d_1)$ with the joining function $r(\cdot)$ embeds into $L_1$ with distortion bounded above by a constant
which depends only on $C_0,C_1, C_2$, and $L$.
\end{theorem}

\begin{proof}
First we observe, that
 for $i=0,1$,  and all $x,y\in M$ we have
\begin{equation}\label{E:gi}
d_i(x,y)\le g_i(x,y)+h(x,y)\le 2d_i(x,y).
\end{equation}

Indeed, by the definitions of $h$ and $g_i$, the rightmost
inequality is clear. If for some $x,y\in M$, either
$g_i(x,y)=d_i(x,y)$ or $h(x,y)=d_i(x,y)$ then the leftmost
inequality also holds. The remaining case is when $g_i(x,y)=r(x)+r(y)<d_i(x,y)$
and $h(x,y)=d_{i+1}(x,y)<d_i(x,y)$, where $i+1=1$ if $i=0$, and $i+1=0$
if $i=1$. In this case, by \eqref{E:Cndr2},
\[d_i(x,y)\le d_{i+1}(x,y)+r(x)+r(y),\]
which proves \eqref{E:gi}.

Denote by $\psi, \vf_0, \vf_1$,   bi-Lipschitz embeddings  into $L_1$ of $(M,h)$, $(M,g_0)$,  and
 $(M,g_1)$, respectively.

Let $m_0\in M$ be such that
\begin{equation}\label{E:Def_m0}
r(m_0)=\min_{x\in M}r(x).
\end{equation}

We define an embedding of $(M\times \mathbb{F}_2,d)$ into
$L_1\oplus_1L_1\oplus_1L_1\oplus_1\mathbb{R}$ by
\begin{equation*}
\begin{split}
G(x,0)&=(\vf_0(x),\vf_1(m_0), \psi(x), r(x))\\
G(x,1)&=(\vf_0(m_0),\vf_1(x), \psi(x), 0).
\end{split}
\end{equation*}

We have
\[\begin{split}\|G(x,0)-G(y,0)\|&=\|\vf_0(x)-\vf_0(y)\|+\|\vf_1(m_0)-\vf_1(m_0)\|
+\|\psi(x)-\psi(y)\|+|r(x)-r(y)|\\&\stackrel{\eqref{E:rhLipschitz-2}}{\asymp}
g_0(x,y)+0+h(x,y) \\
&\stackrel{\eqref{E:gi}}{\asymp}d_0(x,y).
\end{split}\]

\[\begin{split}\|G(x,1)-G(y,1)\|&=\|\vf_0(m_0)-\vf_0(m_0)\|+\|\vf_1(x)-\vf_1(y)\|
+\|\psi(x)-\psi(y)\|+ |0| \\
&\asymp
0+g_1(x,y)+h(x,y)+0\stackrel{\eqref{E:gi}}{\asymp}d_1(x,y).
\end{split}\]

By Lemma~\ref{L:dx0y1-rg}(a) we have
\[d((x,0),(y,1))\asymp h(x,y) +\max\{r(x), r(y)\}.\]

We need to compare this with

\[\begin{split}\|G(x,0)-G(y,1)\|&=\|\vf_0(x)-\vf_0(m_0)\|+\|\vf_1(m_0)-\vf_1(y)\|
+\|\psi(x)-\psi(y)\|+|r(x)-0|\\
&\asymp
g_0(x,m_0)+g_1(m_0,y)+h(x,y)+r(x)\\&\asymp
\min\{d_0(x,m_0),r(x)+r(m_0)\}+\min\{d_1(y,m_0),r(y)+r(m_0)\}\\&{\ \ \ \ }+h(x,y)+r(x)
\\&\stackrel{\eqref{E:Def_m0}\&\eqref{E:rhLipschitz-2}}{\asymp}h(x,y)+\max\{r(x), r(y)\}. \hfill \qedhere\end{split}\]
\end{proof}

\begin{theorem}\label{T:Basic}
Let $M$ be a metric space endowed with two metrics $d_0$ and
$d_1$,   and let $r:M\to (0,\infty)$ be a function such that \eqref{E:Cndr2} and \eqref{E:Cndr1}
are satisfied.

Suppose   that there exists a function $f:M\times M\to[0,\infty)$ and constants $C_1,C_2,C_3,C_4>0$ such that $c_1(M,d_1)\le C_1$, $c_1(M, \min\{f(x,y), r(x)+r(y)\})\le C_2$, and
for all $x,y\in M$
\begin{equation}\label{E:RelBetw_ds}
\frac1{C_3}d_0(x,y)\le d_1(x,y)+f(x,y)\le C_4d_0(x,y).
\end{equation}

Then the  generalized twisted union $(M\times \mathbb{F}_2,d)$
 of $(M,d_0)$  and $(M,d_1)$ with the joining function $r(\cdot)$ embeds into $L_1$ with distortion bounded above by a constant
which depends only on $C_1, C_2, C_3$, and $C_4$.
\end{theorem}

\remove{
\begin{note} It is more convenient to prove Lemma~\ref{SL:dBtwCbs} under
condition \eqref{E:d1lecd0}, which is weaker than
\eqref{E:RelBetw_ds}.
\end{note}}

\begin{proof}
 Similarly as in the proof of \eqref{E:gi}, by a straightforward analysis of two cases, the conditions \eqref{E:RelBetw_ds} and \eqref{E:Cndr2} imply
that for all $x,y\in M$
\begin{equation}\label{E:RelBetw_ds2}
\frac{1}{C_3}d_0(x,y)\le
d_1(x,y)+\min\{f(x,y),r(x)+r(y)\}\le C_4d_0(x,y).
\end{equation}

Let $\vf_1$ and $\psi$ be  bi-Lipschitz embeddings  into $L_1$ of $(M,d_1)$  and
 $(M, \min\{f(x,y), r(x)+r(y)\})$, respectively.
Let $m_0\in M$ be such that
\begin{equation}\label{E:Def_m2} r(m_0)=\min_{x\in M}r(x).
\end{equation}

We
define the embedding of $(M\times \mathbb{F}_2,d)$ into
$L_1\oplus_1L_1\oplus_1\mathbb{R}$ by
\begin{equation}\label{E:GenEmbSig2}
\begin{split}
F(x,0)&=(\vf_1(x), \psi(x), r(x))\\
F(x,1)&=(\vf_1(x),\psi(m),0).
\end{split}
\end{equation}

Then we have
\[\begin{split}\|F(x,0)-F(y,0)\|&=\|\vf_1(x)-\vf_1(y)\|+\|\psi(x)-\psi(y)\|+|r(x)-r(y)|
\\&\asymp
d_1(x,y)+\min\{f(x,y),r(x)+r(y)\}+|r(x)-r(y)|\\
&\stackrel{\eqref{E:RelBetw_ds2}\&\eqref{E:Cndr1}}{\asymp}d_0(x,y).
\end{split}\]

\[\|F(x,1)-F(y,1)\|=\|\vf_1(x)-\vf_1(y)\|+\|\psi(m_0)-\psi(m_0)\|+|0| \asymp
d_1(x,y).\]

The condition \eqref{E:RelBetw_ds} implies that for all $x,y\in M$,
$d_1(x,y)\le C_4d_0(x,y)$, and therefore, by Lemma \ref{SL:dBtwCbs},  we have
\[d((x,0),(y,1))\asymp d_1(x,y)+r(x).\]

We need to compare this with

\[\begin{split}\|F(x,0)-F(y,1)\|&=\|\vf_1(x)-\vf_1(y)\|+\|\psi(x)-\psi(m_0)\|
+|r(x)|\\
&\asymp
d_1(x,y)+\min\{f(x,m_0),r(x)+r(m_0)\}+r(x)\\&\stackrel{\eqref{E:Def_m2}}{\asymp}d_1(x,y)+r(x).
\hfill \qedhere\end{split}\]
\end{proof}

\remove{
\begin{remark} If in Theorem~\ref{T:GenTwPst1}, $d_2(x,y)\le C(d_0(x,y)-d_1(x,y))$ for all $x,y\in M$ and some constant $C$,
then by \eqref{E:Cndr2}, $d_2(x,y)\le C(r(x)+r(y))$, so the
existence of the map $\Psi$ is equivalent to the
$L_1$-embeddability of $(M,d_2)$.
\end{remark}}

Our next result is a direct consequence of either of Theorems~\ref{fixedr-omega}, \ref{T:rLipschitz}, or \ref{T:Basic}.

We note that, by \eqref{E:Ex3.1C}, the space described in Example~\ref{E:ConNR&N} satisfies the assumptions of
Corollary~\ref{fixedr}. Thus we obtain different proofs of the $L_1$-embaddibility of this space, cf.  Corollary~\ref{fixedr-omega-1} above.

\begin{corollary}\label{fixedr}
Let  $r>0$ and $M$ be a metric space endowed with two metrics $d_0$ and
$d_1$,  such that \eqref{E:Cndr2-c}
 is satisfied, and for $i=0,1$, the space $(M,d_i)$ embeds into
$L_1$ with distortion at most $D_i$. Suppose also that there exist
a constant $C>0$  such that for all
$x,y\in M$
\begin{equation}\lb{Kest}
d_1(x,y)\le Cd_0(x,y).
\end{equation}

 Then the $r$-twisted union
$(M\times \mathbb{F}_2,d)$ of $(M,d_0)$  and $(M,d_1)$ 
 embeds into $L_1$ with distortion
bounded by a  constant $D$ that depends only on $D_0, D_1$ and $C$.
\end{corollary}

\begin{proof}
We present three short proofs, each based on Theorems~\ref{fixedr-omega}, \ref{T:rLipschitz}, or \ref{T:Basic}, respectively.

Note that \eqref{Kest} implies that
\[\min\big\{1,1/C\big\}d_1(x,y)\le \min\{d_0(x,y), d_1(x,y)\}\le d_1(x,y).\]

Thus, since $c_1(M,d_1)\le D_1$, we have that
$c_1(M,\min\{d_0(x,y), d_1(x,y)\})<\infty$ and  \eqref{minemb} is satisfied, so Corollary~\ref{fixedr} follows from Theorem~\ref{fixedr-omega}.

Moreover, when the function $r(\cdot)$ is constant, then the functions  $g_i$, for $i=0,1$,
defined in Theorem~\ref{T:rLipschitz} are
$g_i(x,y)= \min\{d_i(x,y), 2r\}$, and thus, by  Corollary~\ref{C:trunc},
$c_1(M,g_i)\le eD_i/(e-1)$. Hence Corollary~\ref{fixedr} follows from Theorem~\ref{T:rLipschitz}.

Similarly, if, in  the notation of Theorem~\ref{T:Basic}, we define the function
$f$ to be equal to $d_0$. Then, by \eqref{Kest}, the inequality
\eqref{E:RelBetw_ds2} is satisfied with $C_3=1$ and $C_4=C+1$. Since $c_1(M,d_0)\le D_0$,
by  Corollary~\ref{C:trunc}, $c_1(M,\min\{d_0(x,y),2r\})\le eD_0/(e-1)$.
Thus, by Theorem~\ref{T:Basic}, Corollary~\ref{fixedr} is satisfied.
\end{proof}

\section{Lower bound on distortion}\label{S:lowerbound}

The goal of this section is to show that the lower bound on
distortion of the union which was found in \cite[Theorem 1.2 and
Section 3]{MM16} for Hilbert space is also valid for $L_1$ and
many other Banach and metric spaces. Also, in some sense, our
proof is simpler.

\begin{definition}\label{D:Stable} {\rm A metric space
$(X,d)$ is called \emph{stable} if for any two bounded sequences
$\{x_n\}$ and $\{y_m\}$ in $X$ and for any two free ultrafilters
$\mju$ and $\mv$ on $\mathbb{N}$
\[\lim_{n,\mju}\lim_{m,\mv}d(x_n,y_m)=\lim_{m,\mv}\lim_{n,\mju}d(x_n,y_m).\]}
\end{definition}

This notion was introduced in the context of Banach spaces by
Krivine and Maurey \cite{KM81}. In the context of metric spaces
this definition was introduced, in a slightly different, but
equivalent form, in \cite[p.~126]{Gar82}. (See \cite{Gue92} for an
account on stable Banach spaces.)

To put our example into context we recall two simple well-known
observations.

\begin{observation}[\cite{KM81}] Hilbert space is stable.
\end{observation}

\begin{proof}
\[\lim_{n,\mju}\lim_{m,\mv}\|x_n-y_m\|^2=\lim_{n,\mju}\|x_n\|^2+\lim_{m,\mv}\|y_m\|^2
-2\lim_{n,\mju}\lim_{m,\mv}\langle x_n,y_m\rangle=
\lim_{m,\mv}\lim_{n,\mju}\|x_n-y_m\|^2. \ \qedhere
\]\end{proof}

\begin{observation}[\cite{Sch38}] The space $L_1(\mathbb{R})$ with the metric $\|x-y\|_1^{\frac12}$ is isometric to a subset of Hilbert space.
\end{observation}

\begin{proof}{\rm(\cite{Nao11})}  We define a map $T:L_1(\mathbb{R})\to
L_\infty(\mathbb{R}\times\mathbb{R})$ by:
\begin{equation*}\label{eq:def T}
T(f)(t,s)\stackrel{\rm def}{=} \left\{
\begin{array}{rl}1& \hbox{if } 0< s\le f(t),\\
-1 & \hbox{if } f(t)<s<0,\\
0& \mathrm{otherwise}.\end{array}\right.
\end{equation*}
For all $f,g\in L_1(\mathbb{R})$ we have:
\begin{equation*}\label{eq:difference}
|T(f)(t,s)-T(g)(t,s)|= \left\{
\begin{array}{ll}1& \hbox{if } g(t)< s\le f(t)\ \mathrm{or}\ f(t)< s\le g(t),\\
0& \mathrm{otherwise}.\end{array}\right.
\end{equation*}
Therefore
\begin{equation*}
\begin{split}
\left\|T(f)-T(g)\right\|_{L_2(\mathbb{R}\times \mathbb{R})}^2
&=\int_{\mathbb{R}} \left(\int_{(g(t),f(t)]{~\rm or~ }(f(t),g(t)]}
1 \ ds\right)dt=\int_{\mathbb{R}}
|f(t)-g(t)|dt\\&=\|f-g\|_{L_1(\mathbb{R})}.\qedhere
\end{split}
\end{equation*}
\end{proof}

\begin{corollary}[\cite{KM81}]\label{C:L1} The space $L_1$ is stable.
\end{corollary}

\begin{example}\label{E:UnstUnion}
Consider the disjoint union of two copies of $\mathbb{N}$:
\[\{\bar 1, \bar 2,\dots,\bar n,\dots\}\cup\,\{\underline
1,\underline 2, \dots,\underline n, \dots\}.\] Endow this union
with the following graph structure: Each $\underline i$ is
adjacent to  $\bar j$ \wtw\  $j\ge i$, and there are no other
edges. Then
\[\lim_{i\to\infty} d(\underline j,\bar i)=1~~ \hbox{ and }~~
\lim_{j\to\infty} d(\underline j,\bar i)=3.\] Observe that
$d(\underline i,\underline j)=2$ and $d(\bar i, \bar j)=2$ for all
$i\ne j$.

Therefore both copies of $\mathbb{N}$ are equilateral and thus
embed isometrically into $\ell_1$. On the other hand, since by
Corollary~\ref{C:L1}, $L_1$ is stable, the distortion of any
embedding of the set constructed in Example \ref{E:UnstUnion} into
$L_1$ is at least $3$.
\end{example}

Of course the same example can be used for any stable metric space
containing an isometric copy of a countable equilateral set. Known
theory  \cite{Gar82,Gue92} implies that, for example,  spaces
$L_p$ for $1\le p<\infty$ satisfy this condition.\medskip

Finally we would like to mention that the distortion $3$ in
Example~\ref{E:UnstUnion} cannot be increased using the same idea.
Namely we prove:

\begin{proposition}\label{P:NoMoreThan3}
Let $\{x_n\}$ and $\{y_m\}$ be two  bounded sequences in a metric
space $X$. Then
\[\lim_{m,\mv}\lim_{n,\mju}d(x_n,y_m)\le 3 \lim_{n,\mju}\lim_{m,\mv}d(x_n,y_m)\]
for any free ultrafilters $\mju$ and $\mv$ on $\mathbb{N}$.
\end{proposition}

\begin{proof} Let $d_{nm}=d(x_n,y_m)$. Passing to subsequences we may assume that the following limits exist: $\lim_{n\to\infty}d_{nm}=S_m$,
$\lim_{m\to\infty}S_m=S$, $\lim_{m\to\infty}d_{nm}=L_n$,
$\lim_{n\to\infty}L_n=L$, where
$S=\lim_{m,\mv}\lim_{n,\mju}d(x_n,y_m)$ and
$L=\lim_{n,\mju}\lim_{m,\mv}d(x_n,y_m)$.

We need to show $S\le 3L$.

Given $\e>0$,   let $M\in\mathbb{N}$ be such that $S_m>S-\ep$ for
all $m\ge M$ and let $N\in\mathbb{N}$ be such that $L_n<L+\ep$ for
all $n\ge N$.

Let $m_N\in\mathbb{N}$ be such that $m_N\ge M$ and
$d(x_N,y_{m_N})<L+\ep$

Let $n_M\in\mathbb{N}$ be such that $n_M\ge N$ and
$d(x_{n_M},y_{m_N})>S-\ep$.

Finally let $f\in\mathbb{\mathbb{N}}$ be such that
$d(x_{N},y_{f})<L+\ep$ and $d(x_{n_M},y_{f})<L+\ep$.

Using the triangle inequality we get
\[S-\ep<d(x_{n_M},y_{m_N})\le d(x_{n_M},y_{f})+d(y_{f},
x_{N})+d(x_N,y_{m_N})<3(L+\ep).\]

Since $\ep>0$ is arbitrary, we get $S\le 3L$.
\end{proof}

\vspace{9mm} {\em Acknowledgements.} We thank Alexandros
Eskenazis for useful information related to the problems studied
in this paper and Gilles Godefroy for his interest in this
research. The first named author gratefully acknowledges the
support by the National Science Foundation grants NSF DMS-1700176
and DMS-1953773, and the hospitality of the Department of
Mathematics of Miami University where this research started.

\end{large}

\begin{small}

\renewcommand{\refname}{
\end{small}

\textsc{Department of Mathematics and Computer Science, St. John's
University, 8000 Utopia Parkway, Queens, NY 11439, USA} \par
  \textit{E-mail address}: \texttt{ostrovsm@stjohns.edu} \par
\smallskip

\textsc{Department of Mathematics, Miami University, Oxford, OH
45056, USA} \par
  \textit{E-mail address}: \texttt{randrib@miamioh.edu} \par

\end{document}